\documentclass[reqno]{amsart}	

\usepackage[english]{babel}
\usepackage{enumerate}
\usepackage{graphicx}
\usepackage{booktabs}
\usepackage{amsaddr, amsmath, amssymb,mathrsfs}
\usepackage{algorithm, algpseudocode, algorithmicx}
\usepackage{stmaryrd}
\usepackage{caption}
\usepackage{subcaption}
\usepackage{here}
\usepackage{pgfplots}
\usepackage{geometry}
\usepackage{color}
\usepackage[onehalfspacing]{setspace}
\allowdisplaybreaks

\pgfplotsset{compat=1.15}
\definecolor{darkbrown}{rgb}{0.57, 0.40, 0.13}

\newcommand{\EE}    {\mathsf E}
\newcommand{\ep}   {e_{\operatorname{plast}}}

\newcommand{\II}    {\mathcal I}
\newcommand{\Jhp}  {\mathcal{J}_{hp}}

\newcommand{\Th}   {\mathcal{T}_h}
\renewcommand{\P}  {\mathcal P}
\newcommand{\Q}    {\mathcal Q}

\newcommand{\CC}   {\mathbb C}
\newcommand{\HH}   {\mathbb H}
\newcommand{\NN}   {\mathbb N}
\newcommand{\PP}   {\mathbb P}
\newcommand{\RR}   {\mathbb R}
\renewcommand{\SS} {\mathbb S}
\newcommand{\VV}   {\mathbb V}

\newcommand{\dd}   {\mathrm{\; d}}
\newcommand{\Div}  {\operatorname{div}}
\newcommand{\dev}  {\operatorname{dev}}
\newcommand{\tr}   {\operatorname{tr}}

\newcommand{\mv}[1]    {\mathfrak{#1}}
\newcommand{\bs}[1]    {\boldsymbol{#1}}
\newcommand{\wt}[1]    {\widetilde{#1}}
\newcommand{\wh}[1]    {\widehat{#1}}
\newcommand{\abs}[1]   {\left\vert #1 \right\vert}
\newcommand{\norm}[1]  {\left\Vert #1 \right\Vert}
\newcommand{\ul}[1]    {\underline{#1}}

\newtheorem{theorem}{Theorem}

\newtheorem{lem}[theorem]   {Lemma}
\newtheorem{coro}[theorem]	{Corollary}

\theoremstyle{definition}

\newtheorem*{defi*}  {Definition}

\newtheorem*{bsp*}	 {Beispiel}

\begin{document}

\renewcommand{\figurename}{\tiny \textbf{Abb.}}

\title[$hp$-Finite Elements for Elastoplasticity]{$\bs{hp}$-Finite Elements for Elastoplasticity}

\author[P.~Bammer, L.~Banz, M.~Sch\"{o}nauer \& A.~Schr\"{o}der]{Patrick Bammer$^{\, 1}$, Lothar Banz$^{\, 2}$, Miriam Sch\"{o}nauer$^{\, 2}$ \& Andreas Schr\"{o}der$^{\, 2}$
}

\address{\small $^{1}$Mathematisches Institut, Universit\"{a}t Bern,\\ Sidlerstr.~5, CH-3012 Bern, Switzerland 
\vspace{0.15cm}\\
$^{2}$Fachbereich Mathematik, Paris Lodron Universit\"{a}t Salzburg,\\ Hellbrunner~Str.~34, 5020 Salzburg, Austria
}

\begin{abstract}
	This article considers a model problem of elastoplasticity with linearly kinematic hardening and presents $hp$-finite element discretizations of two equivalent weak formulations each having their respective advantages. A mixed variational formulation is introduced to resolve the non-differentiablility of the so-called plasticity functional appearing in the weak formulation of the model problem as a variational inequality of the second kind. The discretization of the mixed formulation is then represented as a system of decoupled nonlinear equations which allows the application of an efficient semismooth Newton solver. Finally, an a priori and a posteriori error analysis is given.
\end{abstract}


\maketitle
	

\section{Introduction}
\smallskip

Elastoplasticity with hardening appears in various problems of mechanical engineering, for instance, when modelling the deformation of concrete or metals, see e.g.~\cite{ref:Chen_1988}. Thereby, the specific case of \emph{linearly kinematic hardening} plays an important role.

In this article, a model problem of elastoplasticity with linearly kinematic hardening is considered for which a well established weak formulation is given by a variational inequality of the second kind, see, e.g., \cite{carstensen_1999elastopl, ref:Han_1991, ref:Han_2013}. This variational inequality, however, contains a non-differentiable term which leads to many difficulties not only in the analysis but also in the numerics. We discuss how to circumvent these difficulties by reformulating the variational inequality as a \emph{mixed variational formulation}, in which the non-differentiable term is resolved by introducing a Lagrange multiplier, see, e.g., \cite{ref:Han_1991, ref:Han_1995, ref:Schroeder_2011}. 
The use of \emph{biorthogonal basis functions} for the discretization of the plastic strain and the Lagrange multiplier allows to show the equivalence between a discrete variational inequality, in which the non-differentiable term is approximated, and a discretization of the mixed variational formulation. Furthermore, it enables to decouple the constraints associated with the discrete Lagrange multiplier. This gives the possibility to reformulate the discrete mixed variational formulation in terms of a system of decoupled nonlinear equations which, in fact, simplifies the application of an efficient, superlinearly converging semi-smooth Newton solver, see \cite{ref:Bammer2022Icosahom, ref:Bammer2026newtonSolver}. As the weak solution of problems in elastoplasticity usually does not enjoy high regularity properties (despite smooth data and domain), in general, one has to apply $h$- or $hp$-adaptive finite element schemes to achieve possibly high convergence rates. Since the presented discretizations are conforming in the displacement field and the plastic strain variable but non-conforming in the discrete Lagrange multiplier (except for the lowest-order case) we obtain reduced guaranteed convergence rates but the non-conformity is necessary to receive an implementable discretization. Such reduced guaranteed convergence rates are, however, typical for higher-order mixed methods in the context of variational inequalities, see, e.g., \cite{ref:Banz_2022, ref:Banz_2019, ref:Ovcharova_2017}. We note that in the numerical experiments in \cite{ref:Bammer2022Icosahom, ref:Bammer2025apriori, ref:Bammer2025aposteriori, ref:Bammer2026newtonSolver} we do not observe such a reduction of the convergence rates in practice. \vspace{0.15cm}

The article is structured as follows: In Section~\ref{sec:WeakFormulations}, the model problem is introduced and two equivalent weak formulations (a variational inequality and a mixed variational formulation) are derived. 
In Section~\ref{sec:Discretization}, $hp$-finite element discretizations of both weak formulation are presented. Thereby, the discretization of the mixed variational formulation is based on biorthogonal basis functions leading to a system of decoupled nonlinear equations. 
For numerically solving this system a superlinear converging semismooth Newton solver is proposed in Section~\ref{subsec:algebraic_repr}. The last Section is devoted to an a priori and a posteriori error analysis for the discrete mixed variational formulation. For numerical examples confirming the theoretical statements we refer to \cite{ref:Bammer2022Icosahom, ref:Bammer2025apriori, ref:Bammer2025aposteriori, ref:Bammer2026newtonSolver}.

\medskip

\textbf{Notation} -- For any $n\in\NN$, we write $\ul{n}$ for the finite set $\lbrace 1,\ldots, n\rbrace$. We use standard letters such as $v$ for scalar, fraktur font $(\mv{v})$ for vector and bold letters $(\bs{v})$ for matrix valued quantities, respectively. Furthermore, for $d\in\NN$, we denote the inner product of two vectors $\mv{v} = (v_1,\ldots,v_d)^{\top}, \; \mv{w} = (w_1,\ldots,w_d)^{\top}\in\RR^d$ and matrices $\bs{\sigma} = (\sigma_{ij})_{i,j\in\ul{d}}, \; \bs{\tau} = (\tau_{ij})_{i,j\in\ul{d}}\in\RR^{d\times d}$ by
\begin{align*}
	\mv{v} \cdot \mv{w} := \sum_{i\in\ul{d}} v_i \, w_i, \qquad
	\bs{\sigma} : \bs{\tau} := \sum_{i,j\in\ul{d}} \sigma_{ij} \, \tau_{ij},
\end{align*}
respectively, and write $\, \abs{\mv{v}} := ( \mv{v} \cdot \mv{v} )^{1/2}\,$ and $\, \abs{\bs{\tau}}_F := ( \bs{\tau} : \bs{\tau} )^{1/2} \,$ for the corresponding norms. If we show components of a vector or a matrix this will always be done  with respect to the orthogonal basis $\lbrace \mv{e}_1,\ldots,\mv{e}_d\rbrace$ where $\mv{e}_i$, for $1\leq i\leq d$, denotes the $i$-th Euclidean unit vector in $\RR^d$. Finally, we define the spaces 
\begin{align*}
	\SS_d := \Big\lbrace \bs{\tau}\in\RR^{d\times d} \; ; \; \bs{\tau} = \bs{\tau}^{\top} \Big\rbrace, \qquad
	\SS_{d,0} := \Big\lbrace \bs{\tau}\in\SS_d \; ; \; \tr (\bs{\tau}) = 0 \Big\rbrace,
\end{align*}
where $\, \tr (\bs{\tau}) = \sum_{i\in\ul{d}} \tau_{ii} \,$ is the trace of $\bs{\tau}$, and
\begin{align*}
	\dev(\bs{\tau}) := \bs{\tau} - \frac{1}{d} \, \tr(\bs{\tau}) \, \bs{I}	
\end{align*}
denotes the deviatoric part of a matrix $\bs{\tau}\in\RR^{d\times d}$, where $\bs{I}\in\RR^{d\times d}$ is the identity matrix.

\smallskip


\section{Elastoplasticity with Linearly Kinematic Hardening}\label{sec:WeakFormulations}
\smallskip

\textbf{The Model Problem} -- We consider the behavior of a material body which \emph{reference configuration} (i.e., its undeformed and unstressed state) is given by a bounded polygonal domain $\Omega \subset \RR^d$, $d\in\lbrace 2,3\rbrace$, with Lipschitz-boundary $\Gamma := \partial\Omega$ and outer unit normal $\mv{n}$. 
Let the elastoplastic body $\Omega$ be clamped at a closed Dirichlet boundary part $\Gamma_D\subseteq \Gamma$ of positive surface measure, and set $\Gamma_N := \Gamma \setminus \Gamma_D$. Then, for given volume force $\mv{f} : \Omega\longrightarrow\RR^d$ and surface traction $\mv{g}:\Gamma_N\longrightarrow\RR^d$ the model problem of elastoplasticity with linearly kinematic hardening reads: \emph{Find a displacement field $\mv{u}:\Omega\longrightarrow\RR^d$ and a plastic strain $\bs{p}:\Omega\longrightarrow\SS_{d,0}$ such that }
\begin{subequations}\label{eq:model_problem}
\begin{alignat}{2}
	- \Div \big( \bs{\sigma}(\mv{u},\bs{p}) \big) 
	 &= \mv{f}  &&\text{in } \, \Omega, \\
	\mv{u} 
	 &= \mv{0} &&\text{on } \, \Gamma_D, \\
	\bs{\sigma}(\mv{u},\bs{p}) \, \mv{n} 
	 &= \mv{g} &&\text{on } \, \Gamma_N, \\
	\bs{\sigma} - \HH \, \bs{p} 
	 &\in \partial j(\bs{p}) \qquad &&\text{in } \, \Omega. \label{eq:plastic_flowLawNEW}
\end{alignat}
\end{subequations}
Thereby, the Cauchy stress tensor $\bs{\sigma}:\Omega\to\RR^{d\times d}$ is given by the constitutive equation $\bs{\sigma} = \CC(\bs{\varepsilon}-\bs{p})$, where the strain $\, \bs{\varepsilon}(\mv{u}):=\frac12\bigl(\nabla\mv{u}+(\nabla\mv{u})^\top\bigr)$ splits up into a elastic part $\bs{e}$ and a plastic part $\bs{p}$, i.e.,~$\bs{\varepsilon} = \bs{e} + \bs{p}$. Furthermore, $\CC$ and $\HH$ represent the elasticity and hardening tensor, respectively. Finally, $\partial j(\cdot)$ denotes the subdifferential of the dissipation functional $j(\cdot)$ with $\, j(\bs{q}):=\sigma_y|\bs{q}|_F \,$ for $\bs{q}\in\SS_{d,0}$. Here, $\sigma_y$ is the yield stress in uniaxial tension, which, according to  \cite{ref:Bammer2025apriori, ref:Bammer2025aposteriori, ref:Bammer2026newtonSolver}, we assume to be a positive constant.


\textbf{Weak Formulations of the Model Problem} -- We denote by $H^1(\Omega)$ the usual Sobolev space 
and let $H^1(\Omega,\RR^d)$ and $L^2(\Omega,\RR^d)$ be the space of vector-valued functions $\mv{v} = (v_1,\ldots,v_d)^{\top}$ with components $v_i\in H^1(\Omega)$ and $v_i\in L^2(\Omega)$, respectively. Finally, let $L^2(\Omega,\RR^{d\times d})$ be the space of matrix-valued functions $\bs{q} = (q_{ij})$ with components $q_{ij}\in L^2(\Omega)$ and define the spaces
\begin{align*}
	V := \Big\lbrace \mv{v}\in H^1(\Omega,\RR^d) \; ;\; \gamma\mv{v} = \mv{0} \, \text{ on } \, \Gamma_D \Big\rbrace, \qquad
	Q := L^2(\Omega,\SS_{d,0}),
\end{align*}
where $\gamma : H^1(\Omega,\RR^d) \longrightarrow H^{1/2}(\Gamma_D,\RR^d)$ is the unique trace operator. We simply write $\mv{v}$ for the trace $\gamma\mv{v}$ if its meaning is clear in the following and denote by $(\cdot,\cdot)_{0,\Omega}$ the inner product of $L^2(\Omega)$, $L^2(\Omega,\RR^d)$ and $L^2(\Omega,\RR^{d\times d})$ with associated norm $\norm{\,\cdot\,}_{0,\Omega}$, respectively. We set
\begin{align*}
	\norm{\mv{v}}_{1,\Omega} := \big( \abs{\mv{v}}_{1,\Omega}^2 + \norm{\mv{v}}_{0,\Omega}^2 \big)^{1/2} \qquad
	 \forall \, \mv{v}\in V,
\end{align*}
where $\abs{\mv{v}}_{1,\Omega} := \norm{\bs{\varepsilon}(\mv{v})}_{0,\Omega}$. Note that $\norm{\,\cdot\,}_{1,\Omega}$ is equivalent to the usual Sobolev-norm on $V$ due to \emph{Korn's first inequality}. Therewith, $\VV := V\times Q$ is a Banach space equipped with the norm 
\begin{align*}
    \norm{ (\mv{v},\bs{q}) } := \bigl( \norm{\mv{v}}_{1,\Omega}^2 + \norm{\bs{q}}_{0,\Omega}^2 \bigr)^{1/2} \qquad
     \forall \, (\mv{v},\bs{q})\in\VV.
\end{align*}
Finally, let $V^*$ and $H^{-1/2}(\Gamma_N,\RR^d)$ be the dual space of $V$ and the trace space of $V$ restricted to $\Gamma_N$ with associated dual norm $\norm{\,\cdot\,}_{-\frac{1}{2}, \Gamma_N}$ and denote by $\langle \cdot,\cdot\rangle$ and $\langle\cdot,\cdot\rangle_{\Gamma_N}$ the corresponding duality pairings. \\


\textbf{Variational Inequality Formulation} -- By defining the bilinear form $a : \VV\times \VV\longrightarrow\RR$, the \emph{plasticity functional} $\psi : Q\longrightarrow\RR$ and the linear form $\ell : V\longrightarrow\RR$ by
\begin{align*}
	a\big( (\mv{v},\bs{q}), (\mv{w},\bs{\mu}) \big) 
	 &:= \big( \bs{\sigma}(\mv{v},\bs{q}), \bs{\varepsilon}(\mv{w}) - \bs{\mu} \big)_{0,\Omega} + ( \HH \, \bs{q}, \bs{\mu} )_{0,\Omega}, \\
	\psi(\bs{q}) &:= (\sigma_y, \abs{\bs{q}}_F)_{0,\Omega}, \\
	\ell(\mv{v}) &:= \langle \mv{f},\mv{v}\rangle + \langle \mv{g},\mv{v}\rangle_{\Gamma_N},
\end{align*}
where $\mv{f}\in V^*$ and $\mv{g}\in H^{-1/2}(\Gamma_N,\RR^d)$, we arrive at the well established weak formulation of \eqref{eq:model_problem}, see, e.g., \cite{carstensen_1999elastopl, ref:Han_1991, ref:Han_2013}: \emph{Find a pair $(\mv{u},\bs{p})\in\VV$ such that}
\begin{align}\label{eq:viq2}
	a\big( (\mv{u},\bs{p}), (\mv{v} - \mv{u}, \bs{q} - \bs{p}) \big) + \psi(\bs{q}) - \psi(\bs{p}) \geq \ell(\mv{v} - \mv{u}) \qquad
	 \forall \, (\mv{v},\bs{q})\in\VV.
\end{align}
%


In order to guarantee the unique existence of a weak solution let the elasticity tensor $\CC$ and the hardening tensor $\HH$ with components $\CC_{ijkl},\HH_{ijkl}\in L^{\infty}(\Omega)$ satisfy the symmetry properties $\CC_{ijkl} = \CC_{jilk} = \CC_{klij}$ and $\HH_{ijkl} = \HH_{jilk} = \HH_{klij}$ for all $i,j,k,l\in\ul{d}$ and, in addition, be uniformly elliptic, i.e., there are constants $c_e,c_h>0$ such that 
%
\begin{align*}
	(\CC \, \bs{q}) : \bs{q} \geq c_e \abs{\bs{q}}_F^2, \quad
	(\HH \, \bs{q}) : \bs{q} \geq c_h \abs{\bs{q}}_F^2 \qquad
	 \forall \, \bs{q}\in Q.
\end{align*}
Therewith, $a(\cdot,\cdot)$ is a symmetric, continuous and $\VV$-elliptic bilinear form on $\VV$, see \cite{ref:Bammer2024Diss, ref:Han_2013}. Furthermore, the plasticity functional $\psi(\cdot)$ is convex, Lipschitz-continuous 
and, see Lemma~1, subdifferentiable.

\begin{lem}
	For arbitrary $\bs{\mu}\in Q$ there exists an element $D\psi(\bs{\mu})$ in $Q^*$, the dual space of $Q$, such that 
	\begin{align*}
		\psi(\bs{q}) \geq \psi(\bs{\mu}) + \langle D\psi(\bs{\mu}), \bs{q} - \bs{\mu} \rangle \qquad
		 \forall \, \bs{q}\in Q.
	\end{align*}
\end{lem}

\begin{proof}
	For $\bs{\mu}\in Q$ let $A_{\bs{\mu}} := \lbrace \mv{x}\in\Omega \; ;\; \bs{\mu}(\mv{x})\neq \bs{0} \rbrace$. Note that 
	\begin{align*}
		\frac{\bs{\mu}}{\abs{\bs{\mu}}_F} : \bs{q} \leq \abs{ \frac{\bs{\mu}}{\abs{\bs{\mu}}_F} }_F \, \abs{\bs{q}}_F
		 = \abs{\bs{q}}_F \qquad
		  \forall \, \bs{q}\in Q
	\end{align*}
	by the Cauchy-Schwarz inequality. Thus, by defining 
    \begin{align*}
        \wt{\bs{q}} := \tfrac{\sigma_y}{\abs{\bs{\mu}}_F} \, \bs{\mu} \quad \text{on } \, A_{\bs{\mu}}
         \qquad\text{and}\qquad
        \wt{\bs{q}} := \bs{0} \quad \text{on }\, \Omega\setminus A_{\bs{\mu}},
    \end{align*}
	we deduce from the above inequality and since $\sigma_y>0$ that
	\begin{align*}
		\psi(\bs{q}) - \psi(\bs{\mu})
		 &= \int_{A_{\bs{\mu}}} \sigma_y \, \big( \abs{\bs{q}}_F - \abs{\bs{\mu}}_F \big) \dd\mv{x} + \int_{\Omega\setminus A_{\bs{\mu}}} \sigma_y \, \abs{\bs{q}}_F \dd\mv{x} \\
		 &\geq \int_{A_{\bs{\mu}}} \sigma_y \bigg( \frac{\bs{\mu}}{\abs{\bs{\mu}}_F} : \bs{q} - \frac{\bs{\mu}}{\abs{\bs{\mu}}_F} : \bs{\mu} \bigg) \dd\mv{x} \\
		 &= \int_{A_{\bs{\mu}}} \frac{\sigma_y}{\abs{\bs{\mu}}_F} \, \bs{\mu} : (\bs{q} - \bs{\mu}) \dd\mv{x}
		 =  \int_\Omega \wt{\bs{q}} : (\bs{q} - \bs{\mu}) \dd \mv{x}.
	\end{align*}
	By Riesz' representation theorem $\wt{\bs{q}}$ uniquely determines an element $D\psi(\bs{\mu})\in Q^{\ast}$ that satisfies the desired inequality.
\end{proof}

The properties of $a(\cdot,\cdot)$, $\psi(\cdot)$ and $\ell(\cdot)$ guarantee that the energy functional 
\begin{align*}
	\EE : \VV\longrightarrow\RR, \qquad
	\EE(\mv{v},\bs{q}) := \frac{1}{2} \, a\big( (\mv{v},\bs{q}), (\mv{v},\bs{q}) \big) + \psi(\bs{q}) - \ell(\mv{v})
\end{align*}
is coercive, convex and subdifferentiable, see, e.g.~\cite{ref:Bammer2024Diss}. Moreover, $(\mv{u},\bs{p})\in\VV$ satisfies \eqref{eq:viq2} if and only if it solves the minimization problem: \emph{Find a pair $(\mv{u},\bs{p})\in\VV$ such that} 
\begin{align*}
	\EE(\mv{u},\bs{p}) \leq \EE(\mv{v},\bs{q}) \qquad
	 \forall \, (\mv{v},\bs{q})\in\VV
\end{align*}
see, e.g., \cite{ref:Han_1991, ref:Han_2013}. Since this problem has a unique solution by \cite{ref:Ekeland_1999} there exists a unique solution $(\mv{u},\bs{p})\in\VV$ of \eqref{eq:viq2}. \\


\textbf{Mixed Variational Formulation} -- To circumvent problems caused by the non-differentiability of $\psi(\cdot)$, we introduce a Lagrange multiplier resolving the Frobenius norm in the definition of $\psi(\cdot)$. We therefore define the nonempty, closed and convex set
\begin{align}\label{eq:lambda_strong}
	\Lambda := \Big\lbrace \bs{\mu}\in Q \; ; \; \abs{\bs{\mu}}_F \leq \sigma_y \, \text{ a.e.~in } \Omega \Big\rbrace,
\end{align}
which can equivalently be represented as
\begin{align}\label{eq:lambda_weak}
	\Lambda = \Big\lbrace \bs{\mu}\in Q \; ; \; (\bs{\mu},\bs{q})_{0,\Omega} \leq \psi(\bs{q}) \quad \forall \, \bs{q}\in Q \Big\rbrace,
\end{align}
see \cite{ref:Bammer2025apriori}.  
As $(\bs{\mu},\bs{q})_{0,\Omega} \leq \psi(\bs{q})$ for arbitrary $\bs{q}\in Q$ and any $\bs{\mu}\in\Lambda$ the representation \eqref{eq:lambda_weak} implies 
\begin{align}\label{eq:my_inequality}
	\sup_{\bs{\mu}\in\Lambda} ( \bs{\mu},\bs{q})_{0,\Omega} \leq \psi(\bs{q}).
\end{align}
	Since for $\wt{\bs{\mu}} := \frac{\sigma_y}{\abs{\bs{q}}_F} \, \bs{q}\in Q$ if $\bs{q}\neq\bs{0}$ we have
	\begin{align*}
		(\wt{\bs{\mu}},\bs{q})_{0,\Omega} 
		 = \int_\Omega \frac{\sigma_y}{\abs{\bs{q}}_F} \, \bs{q} : \bs{q} \dd\mv{x}
		 = \int_\Omega \sigma_y \, \abs{\bs{q}}_F \dd\mv{x}
		 =\psi(\bs{q}) \qquad
		  \forall \, \bs{q}\in Q\setminus\lbrace\bs{0}\rbrace,
	\end{align*}
	it follows that $\wt{\bs{\mu}}\in\Lambda$ (for $\bs{q}=\bs{0}$ we take $\wt{\bs{\mu}} := \bs{0}$). Thus, the inequality \eqref{eq:my_inequality} together with 
	\begin{align*}
		(\wt{\bs{\mu}},\bs{q})_{0,\Omega}~\leq~\sup_{\bs{\mu}\in\Lambda} ( \bs{\mu},\bs{q})_{0,\Omega}
	\end{align*}
	gives
	\begin{align*}
		\psi(\bs{q}) = \sup_{\bs{\mu}\in\Lambda} ( \bs{\mu},\bs{q})_{0,\Omega} \qquad
		 \forall \, \bs{q}\in Q.
\end{align*}		
This consideration motivates to introduce the mixed variational formulation: \emph{Find a triple $(\mv{u},\bs{p},\bs{\lambda})\in\VV\times\Lambda$ such that}
\begin{subequations}\label{eq:mixed_varF}
\begin{alignat}{2}
	a\big( (\mv{u},\bs{p}), (\mv{v},\bs{q}) \big) + (\bs{\lambda},\bs{q})_{0,\Omega} &= \ell(\mv{v}) \qquad 
	 &&\forall \, (\mv{v},\bs{q})\in\VV, \label{eq:weak_vareq} \\
	(\bs{\mu} - \bs{\lambda}, \bs{p})_{0,\Omega} &\leq 0
	 &&\forall \, \bs{\mu}\in\Lambda.
\end{alignat}
\end{subequations}
Since there exists a unique solution $(\mv{u},\bs{p})\in\VV$ of \eqref{eq:viq2} the first part of the next result guarantees the unique existence of a solution $(\mv{u},\bs{p},\bs{\lambda})\in\VV\times\Lambda$ of \eqref{eq:mixed_varF}.

\begin{theorem}[Equivalence of the Weak Formulations -- {\cite[Thm.1, Lem.3]{ref:Bammer2025apriori}}]
\quad
	\begin{enumerate}
		\item[\bf 1.] If $(\mv{u},\bs{p})\in\VV$ satisfies \eqref{eq:viq2}, then $(\mv{u},\bs{p},\bs{\lambda})$ with 
\begin{align}\label{defi:special_lambda}
	\bs{\lambda} := \dev\big( \bs{\sigma}(\mv{u},\bs{p}) - \HH\, \bs{p} \big)
\end{align}
		solves \eqref{eq:mixed_varF}. Conversely, if $(\mv{u},\bs{p},\bs{\lambda})\in\VV\times\Lambda$ solves \eqref{eq:mixed_varF}, then $(\mv{u},\bs{p})$ satisfies \eqref{eq:viq2} and the identity \eqref{defi:special_lambda} holds true.
		
		\item[\bf 2.] The solution $(\mv{u},\bs{p},\bs{\lambda})\in\VV\times\Lambda$ of \eqref{eq:mixed_varF} depends Lipschitz-continuously on the data $\mv{f}, \mv{g}$ and $\sigma_y$, i.e., if $(\mv{u}_i,\bs{p}_i,\bs{\lambda}_i)\in\VV\times \Lambda$, for $i=1,2$, denotes the solution of \eqref{eq:mixed_varF} for the given data $\mv{f}_i, \mv{g}_i$ and $\sigma_{y,i}$, respectively, there are constants $c, c_{\operatorname{tr}} >0$ such that
		\begin{align*}
			&\norm{ (\mv{u}_2 - \mv{u}_1, \bs{p}_2 - \bs{p}_1) } + \norm{ \bs{\lambda}_2 - \bs{\lambda}_1 }_{0,\Omega} \\ 
			& \hspace{2cm} \leq c \, \Big( \Vert \sigma_{y,2} - \sigma_{y,1} \Vert_{0,\Omega} + \Vert \mv{f}_2 - \mv{f}_1 \Vert_{V^{\ast}} + c_{\operatorname{tr}} \, \Vert \mv{g}_2 - \mv{g}_1 \Vert_{-\frac{1}{2},\Gamma_N} \Big).
		\end{align*}
	\end{enumerate}
\end{theorem}

\medskip


\section{$hp$-FE Discretization using Biorthogonal Basis Functions}\label{sec:Discretization}
\smallskip

Let $\wh{T} := [-1,1]^d$ be the reference element and $\Th$ a locally quasi-uniform finite element mesh of $\Omega$ consisting of convex and shape regular quadrilaterals ($d=2$) or hexahedrons ($d=3$). For $T\in\Th$, let $\, \mv{M}_T:\wh{T}\longrightarrow T \,$ be the bi/tri-linear bijective mapping and set $h := (h_T)_{T\in\Th}$, $p := (p_T)_{T\in\Th}$ with local element size $h_T$ and polynomial degree $p_T$. We assume the local polynomial degrees of neighbouring elements to be comparable and refer to \cite{ref:melenk_2005} for details on quasi-uniformity and comparable polynomial degrees. Next, we introduce the conforming $hp$-finite element spaces
\begin{align*}
	V_{hp} &:= \Big\lbrace \mv{v}_{hp} \in V \; ; \; \mv{v}_{hp \, \vert \, T} \circ \mv{M}_T \in\big( \PP_{p_T}(\wh{T})\big)^d \quad \forall \, T\in\Th \Big\rbrace,\\
	Q_{hp} &:= \Big\lbrace \bs{q}_{hp}\in Q_{hp} \; ; \; \bs{q}_{hp \, \vert \, T} \circ \mv{M}_T \in \big(\PP_{p_T-1}(\wh{T})\big)^{d\times d} \quad \forall \, T\in\Th \Big\rbrace
\end{align*}
and set $\VV_{hp} := V_{hp}\times Q_{hp}$, where $\PP_{p_T}(\wh{T})$ denotes the space of polynomials up to degree $p_T$ on $\wh{T}$. 
For $T\in\Th$, let $n_T := p_T^d$ and denote by $\hat{\mv{x}}_{k,T}\in\wh{T}$ the tensor product Gauss quadrature points on $\wh{T}$ with corresponding weights $\omega_{k,T}>0$ for $k\in\underline{n_T}$. Moreover, for $T\in\Th$ and $k\in\ul{n_T}$, we write $\wh{\phi}_{k,T}\in\PP_{p_T-1}(\wh{T})$ for the Lagrange basis functions on $\wh{T}$ associated with these points, i.e.,
\begin{align*}
	\wh{\phi}_{k,T}(\hat{\mv{x}}_{l,T}) = \delta_{kl} \qquad
	 \forall \, k,l\in\ul{n_T},
\end{align*}
where $\delta_{kl}$ is the Kronecker delta. For $N := \sum_{T\in\Th} n_T$, we define $\phi_1,\ldots,\phi_N : \Omega\longrightarrow\RR$ elementwise by 
\begin{align*}
	\phi_{\zeta(k,T') \, \vert \, T} := \begin{cases}
		\wh{\phi}_{k,T'}\circ\mv{M}_T^{-1}, & \text{if } \, T=T', \\
		0, & \text{if } \, T\neq T',
	\end{cases} \qquad
	 \forall \, T,T'\in\Th \quad \forall \, k\in\ul{n_T},
\end{align*}
where $\zeta : \lbrace (k,T) \; ; \; k\in\ul{n_T}, \; T\in\Th \rbrace\longrightarrow\lbrace 1,\ldots,N\rbrace$ is one-to-one. Finally, we choose \emph{biorthogonal} basis functions $\varphi_1,\ldots,\varphi_N:\Omega\longrightarrow\RR$ that are uniquely determined by
\begin{align*}
	\varphi_{\zeta(k,T) \, \vert \, T} \circ \mv{M}_T \in \PP_{p_T-1}(\wh{T}) \quad
	   \forall \, k\in\ul{n_T}, \; \forall \, T\in\Th
	\qquad\text{and}\qquad
	(\phi_i,\varphi_j)_{0,\Omega} = \delta_{ij} \, (\phi_i, 1)_{0,\Omega} \quad
	   \forall \, i,j\in\ul{N}.
\end{align*}
Note that $\, \operatorname{supp}(\phi_i) = \operatorname{supp}(\varphi_i) \,$ for $i\in\ul{N}$ and both ($\phi_i$ and $\varphi_i$) form a basis of 
\begin{align*}
	W_{hp} := \Big\lbrace q_{hp}\in L^2(\Omega) \; ; \; q_{hp \, \vert \, T} \circ \mv{M}_T\in\PP_{p_T-1}(\wh{T}) \; \forall \, T\in\Th \Big\rbrace
\end{align*}
Thus,
\begin{align}\label{eq:repres_Qhp}
	Q_{hp} = \bigg\lbrace \sum_{i\in\ul{N}} \bs{q}_i \, \phi_i \; ; \; \bs{q}_i\in\SS_{d,0} \; \forall \, i\in\ul{N} \bigg\rbrace 
	= \bigg\lbrace \sum_{i\in\ul{N}} \bs{\mu}_i \, \varphi_i \; ; \; \bs{\mu}_i\in\SS_{d,0} \; \forall \, i\in\ul{N} \bigg\rbrace.
\end{align}


\textbf{Discretization of the Variational Inequality} -- To obtain an implementable approximation of $\psi(\cdot)$ we introduce the mesh depended quadrature rule $\Q_{hp}(\cdot) := \sum_{T\in\Th} \Q_{hp,T}(\cdot)$, where, for $T\in\Th$, the local quantities $\Q_{hp,T}(\cdot)$ are given by 
\begin{align*}
	\Q_{hp,T}(f) := \begin{cases}
		\abs{T} \, f\big(\mv{M}_T(\mv{0})\big), & \text{if } \, p_T = 1, \\
		\sum_{k=1}^{n_T} \omega_{k,T} \, | \det \nabla\mv{M}_T(\hat{\mv{x}}_{k,T}) | \, f\big( \mv{M}_T(\hat{\mv{x}}_{k,T}) \big), & \text{if } \, p_T \geq 2,
	\end{cases}
\end{align*}
$\abs{T}$ being the $d$-dimensional Lebesque-measure of $T$. Therewith, we let 
%
\begin{align*}
    \psi_{hp} : Q_{hp}\longrightarrow\RR, \qquad
    \psi_{hp}(\bs{q}_{hp}) := \Q_{hp}\big( \sigma_y \, \vert\bs{q}_{hp}\vert_F \big),
\end{align*}
cf.~\cite{ref:Bammer2026newtonSolver}. An approximation of \eqref{eq:viq2} therefore reads: \emph{Find a pair $(\mv{u}_{hp}, \bs{p}_{hp})\in\VV_{hp}$ such that}
\begin{align}\label{eq:discrete_viq2}
	a\big( (\mv{u}_{hp},\bs{p}_{hp}), (\mv{v}_{hp} - \mv{u}_{hp}, \bs{q}_{hp} - \bs{p}_{hp}) \big) + \psi_{hp}(\bs{q}_{hp}) - \psi_{hp}(\bs{p}_{hp}) \geq \ell(\mv{v}_{hp} - \mv{u}_{hp}) \qquad
	 \forall \, (\mv{v}_{hp},\bs{q}_{hp})\in\VV_{hp}.
\end{align}
%


\begin{lem}[Subdifferentiability of $\psi_{hp}(\cdot)$ -- {\cite[Lem.1, Lem.2]{ref:Bammer2025aposteriori}}]
	\quad
	\begin{enumerate}
		\item[\bf 1.] By representing $\bs{q}_{hp}\in Q_{hp}$ as $\, \bs{q}_{hp} = \sum_{i\in\ul{N}} \bs{q}_i \, \phi_i\, $, we find that
		\begin{align*}
			\psi_{hp}(\bs{q}_{hp}) = \sum_{i\in\ul{N}} \vert\bs{q}_i\vert_F \, (\sigma_y, \phi_i)_{0,\Omega}.
		\end{align*}				
		\item[\bf 2.] For any $\bs{\mu}_{hp}\in Q_{hp}$ there exists a $D\psi_{hp}(\bs{\mu}_{hp})$ in $Q_{hp}^*$, the dual space of $Q_{hp}$, such that
		\begin{align*}
			\psi_{hp}(\bs{q}_{hp}) \geq \psi_{hp}(\bs{\mu}_{hp}) + \langle D\psi_{hp}(\bs{\mu}_{hp}), \bs{q}_{hp} - \bs{\mu}_{hp} \rangle \qquad
			 \forall \, \bs{q}_{hp}\in Q_{hp}.
		\end{align*}
	\end{enumerate}
\end{lem}

By using Lemma~2 it is shown in \cite{ref:Bammer2025aposteriori} that the discrete energy $\EE_{hp} : \VV_{hp}\longrightarrow\RR$ with
\begin{align*}
	\EE_{hp}(\mv{v}_{hp},\bs{q}_{hp}) := \frac{1}{2} \, a\big( (\mv{v}_{hp},\bs{q}_{hp}), (\mv{v}_{hp},\bs{q}_{hp}) \big) + \psi_{hp}(\bs{q}_{hp}) - \ell(\mv{v}_{hp})
\end{align*}
is coercive, convex and subdifferentiable so that the discrete minimization problem: \emph{Find a pair $(\mv{u}_{hp},\bs{p}_{hp})\in\VV_{hp}$ such that}
\begin{align*}
	\EE_{hp}(\mv{u}_{hp},\bs{p}_{hp}) \leq \EE_{hp}(\mv{v}_{hp},\bs{q}_{hp}) \qquad
	 \forall \, (\mv{v}_{hp},\bs{q}_{hp})\in\VV_{hp}
\end{align*}
has a unique solution by \cite{ref:Ekeland_1999}. Since that minimization problem is equivalent to the discrete variational inequality (similar as in Section~\ref{sec:WeakFormulations}) we find that \eqref{eq:discrete_viq2} has a unique solution $(\mv{u}_{hp},\bs{p}_{hp})\in\VV_{hp}$, see \cite[Thm.3]{ref:Bammer2025aposteriori}. \\


\textbf{Discretization of the Mixed Variational Formulation} -- To discretize $\Lambda$ we can either use the representation \eqref{eq:lambda_strong} \emph{or} \eqref{eq:lambda_weak}, in general, leading to two different dicretizations. Under a slight limitation on the mesh elements' shape, however, they are equivalent, see Section~\ref{sec:aprioriEst}. Discretizing \eqref{eq:lambda_weak} leads to the non-empty, convex and closed set
\begin{align*}
	\Lambda_{hp} := \Big\lbrace \bs{\mu}_{hp}\in Q_{hp} \; ; \; (\bs{\mu}_{hp}, \bs{q}_{hp})_{0,\Omega} \leq \psi_{hp}(\bs{q}_{hp}) \quad \forall \, \bs{q}_{hp}\in Q_{hp} \Big\rbrace.
\end{align*}

A discrete version of \eqref{eq:mixed_varF} thus reads: \emph{Find a triple $(\mv{u}_{hp}, \bs{p}_{hp}, \bs{\lambda}_{hp})\in\VV_{hp}\times \Lambda_{hp}$ s.t.}
\begin{subequations}\label{eq:discrete_MVF}
\begin{alignat}{2}
	a\big( (\mv{u}_{hp},\bs{p}_{hp}), (\mv{v}_{hp},\bs{q}_{hp}) \big) + (\bs{\lambda}_{hp},\bs{q}_{hp})_{0,\Omega} &= \ell(\mv{v}_{hp}) \qquad
	 &&\forall \, (\mv{v}_{hp},\bs{q}_{hp})\in\VV_{hp}, \\
	(\bs{\mu}_{hp} - \bs{\lambda}_{hp}, \bs{p}_{hp})_{0,\Omega} &\leq 0
	 &&\forall \, \bs{\mu}_{hp}\in\Lambda_{hp}. \label{eq:disc_MVF_inequalityConst}
\end{alignat}
\end{subequations}
The second part of the next theorem guarantees the existence and uniqueness of the first two components of a solution $(\mv{u}_{hp},\bs{p}_{hp},\bs{\lambda}_{hp})\in\VV_{hp}\times \Lambda_{hp}$ of \eqref{eq:discrete_MVF}. The remaining uniqueness of the discrete Lagrange multiplier follows from the \emph{discrete inf-sup condition}, stated in the first part of the theorem (and noting that $\Lambda_{hp}\subseteq Q_{hp}$).

\begin{theorem}[Equivalence of Discretizations -- {\cite[Lem.4]{ref:Bammer2025apriori}\cite[Thm.4]{ref:Bammer2025aposteriori}}]\label{thm:equivalence_discreteF}
	\quad
	\begin{enumerate}
		\item[\bf 1.] It holds that
		\begin{align*}
			\sup_{\substack{ (\mv{v}_{hp},\bs{q}_{hp})\in\VV_{hp} \\ \Vert(\mv{v}_{hp},\bs{q}_{hp}) \Vert \neq 0 }} \frac{(\bs{\mu}_{hp}, \bs{q}_{hp})_{0,\Omega}}{\Vert (\mv{v}_{hp},\bs{q}_{hp})\Vert }
			= \Vert \bs{\mu}_{hp} \Vert_{0,\Omega} \qquad
			 \forall \, \bs{\mu}_{hp}\in Q_{hp}.
		\end{align*}
		\item[\bf 2.] If $(\mv{u}_{hp},\bs{p}_{hp})\in\VV_{hp}$ solves \eqref{eq:discrete_viq2} then, $(\mv{u}_{hp},\bs{p}_{hp},\bs{\lambda}_{hp})$ with
		\begin{align}\label{eq:rep_discreteLM}
			\bs{\lambda}_{hp} := \P_{hp}\Bigl( \dev\big( \bs{\sigma}(\mv{u}_{hp},\bs{p}_{hp}) - \HH\, \bs{p}_{hp} \big) \Bigr)
		\end{align}
		is a solution of \eqref{eq:discrete_MVF}, where $\P_{hp} : Q\longrightarrow Q_{hp}$ is the standard $L^2$-projection operator. Conversely, if $(\mv{u}_{hp},\bs{p}_{hp},\bs{\lambda}_{hp})\in\VV_{hp}\times\Lambda_{hp}$ solves \eqref{eq:discrete_MVF} then, $(\mv{u}_{hp},\bs{p}_{hp})$ is a solution of \eqref{eq:discrete_viq2} and the identity \eqref{eq:rep_discreteLM} holds true.
	\end{enumerate}
\end{theorem}



\section{Semi-Smooth Newton Solver}\label{subsec:algebraic_repr}
\smallskip

To decouple the constraint in $\Lambda_{hp}$ and \eqref{eq:disc_MVF_inequalityConst}, we recall \eqref{eq:repres_Qhp} and define the quantities $D_i := (\phi_i,1)_{0,\Omega}\, $ and $\, \sigma_i := D_i^{-1} \, (\sigma_y, \phi_i)_{0,\Omega}\,$ for all $i\in\ul{N}$.
Since $\phi_1,\ldots,\phi_N$ represent the Gauss-Legendre-Lagrange basis functions we have $D_i > 0$ for all $i\in\ul{N}$ and as $\sigma_y>0$ it also holds that $\sigma_i >0$ for all $i\in\ul{N}$.

\begin{theorem}[Equivalent Representation of $\Lambda_{hp}$ -- {\cite[Thm.2]{ref:Bammer2022Icosahom}}]\label{thm:decoupling_lambdahp}
	\quad
	\begin{enumerate}
		\item[\bf 1.] The set $\Lambda_{hp}$ of admissible discrete Lagrange multipliers can be represented as
		\begin{align*} 
			\Lambda_{hp} = \bigg\lbrace \sum_{i\in\ul{N}} \bs{\mu}_i \, \varphi_i \; ; \; \bs{\mu}_i\in\SS_{d,0} \, \text{ and } \, \vert\bs{\mu}_i\vert_F \leq \sigma_i \quad \forall \, i\in\ul{N} \bigg\rbrace.
		\end{align*}
		\item[\bf 2.] By writing $\bs{p}_{hp}\in Q_{hp}$ and $\bs{\lambda}_{hp}\in\Lambda_{hp}$ as 
		\begin{align*}
			\bs{p}_{hp} = \sum_{i\in\ul{N}} \bs{p}_i \, \phi_i, \qquad
			\bs{\lambda}_{hp} = \sum_{i\in\ul{N}} \bs{\lambda}_i \, \varphi_i,
		\end{align*}				
		$\bs{\lambda}_{hp}$ satisfies \eqref{eq:disc_MVF_inequalityConst} if and only if $\,\bs{\lambda}_i : \bs{p}_i = \sigma_i \, \vert\bs{p}_i\vert_F \,$ for all $i\in\ul{N}$.
	\end{enumerate}
\end{theorem}

\begin{coro}
	Let $(\mv{u}_{hp},\bs{p}_{hp},\bs{\lambda}_{hp})\in\VV_{hp}\times\Lambda_{hp}$ be the unique solution of \eqref{eq:discrete_MVF}. By representing $\bs{p}_{hp}$ and $\bs{\lambda}_{hp}$ as in Theorem~\ref{thm:decoupling_lambdahp}, for $i\in\ul{N}$, we find that:
	\begin{enumerate}
		\item[\emph{(i)}] If $\abs{\bs{\lambda}_i}_F < \sigma_i$, then $\bs{p}_i = \bs{0}$.
		\item[\emph{(ii)}] If $\abs{\bs{\lambda}_i}_F = \sigma_i$, then there exists a constant $c\geq 0$ such that $\bs{p}_i = c \, \bs{\lambda}_i$.
	\end{enumerate}
\end{coro}

\begin{proof}
	Let $i\in\ul{N}$. By using Theorem~\ref{thm:decoupling_lambdahp} and the Cauchy-Schwarz inequality, we have
    \begin{align*}
        \sigma_i \, \vert\bs{p}_i\vert_F
		 = \bs{\lambda}_i : \bs{p}_i
		 \leq \vert\bs{\lambda}_i\vert_F \vert\bs{p}_i\vert_F
		 \leq \sigma_i \, \vert\bs{p}_i\vert_F,
    \end{align*}
which implies that $\, \bs{\lambda}_i : \bs{p}_i 
		 = \vert\bs{\lambda}_i\vert_F \vert\bs{p}_i\vert_F
		 = \sigma_i \, \vert\bs{p}_i\vert_F \,$ for all $i\in\ul{N}$.	
	If $\, \abs{\bs{\lambda}_i}_F < \sigma_i$, the second equality gives $\bs{p}_i = \bs{0}$ whereas if $\abs{\bs{\lambda}_i}_F = \sigma_i$, there exists a constant $c\geq 0$ such that $\bs{p}_i = c \, \bs{\lambda}_i$ since equality holds true in the Cauchy-Schwarz inequality if and only if the elements are linearly dependent.
\end{proof}

For $i\in\ul{N}$ and some $\varrho > 0$ we define the nonlinear functions $\bs{\chi}_{i,\varrho} : \SS_{d,0}\times\SS_{d,0}\longrightarrow\SS_{d,0}$
\begin{align}\label{eq:defi_chi_irho}
	\bs{\chi}_{i,\varrho}(\bs{q}_i,\bs{\mu}_i)
	 := \max \big\lbrace \sigma_i, \vert\bs{\mu}_i + \varrho\, \bs{q}_i\vert_F \big\rbrace \, \bs{\mu}_i - \sigma_i \, (\bs{\mu}_i + \varrho\, \bs{q}_i),
\end{align}
where $\bs{q}_i,\bs{\mu}_i$ are the coefficients of $\bs{q}_{hp}\in Q_{hp}$ and $\bs{\mu}_{hp}\in\Lambda_{hp}$ represented as 
\begin{align*}
	\bs{q}_{hp} = \sum_{i\in\ul{N}} \bs{q}_i \, \phi_i, \qquad
	\bs{\mu}_{hp} = \sum_{i\in\ul{N}} \bs{\mu}_i \, \varphi_i.
\end{align*}

\begin{theorem}[{\cite[Thm.3]{ref:Bammer2022Icosahom}}]\label{thm:ncp_function}
	Let $(\mv{u}_{hp},\bs{p}_{hp},\bs{\lambda}_{hp})\in\VV_{hp}\times\Lambda_{hp}$ be the solution of \eqref{eq:discrete_MVF} and $\bs{p}_{hp}, \bs{\lambda}_{hp}$ written as in Theorem~\ref{thm:decoupling_lambdahp}. Then, the discrete Lagrange multiplier $\bs{\lambda}_{hp}\in\Lambda_{hp}$ satisfies \eqref{eq:disc_MVF_inequalityConst} if and only if
	\begin{align*}
		\bs{\chi}_{i,\varrho}(\bs{p}_i,\bs{\lambda}_i) = \bs{0} \qquad
		 \forall \, i\in\ul{N}.
	\end{align*}		
\end{theorem}

Theorem~\ref{thm:ncp_function} allows to rewrite \eqref{eq:discrete_MVF} as a system of decoupled nonlinear equations. To this end, we first choose an orthonormal basis $\mv{B}_{\SS} = \lbrace \bs{\Phi}_1,\ldots,\bs{\Phi}_L\rbrace$ of $\SS_{d,0}$ (with respect to the Frobenius inner product), where $L := \frac{1}{2} \, (d-1) (d+2)$: If $d = 2$, we take 
\begin{align*}
	\bs{\Phi}_1 := \frac{1}{\sqrt{2}} \, \left( \begin{matrix} 1 & 0 \\ 0 & -1 \end{matrix} \right), \qquad
	\bs{\Phi}_2 := \frac{1}{\sqrt{2}} \, \left( \begin{matrix} 0 & 1 \\ 1 & 0 \end{matrix} \right)
\end{align*}
and if $d = 3$, we take 
\begin{align*}
	\bs{\Phi}_1 
            &:= \frac{1}{\sqrt{2}}  \, \left(\begin{matrix} 1 & 0 & 0 \\ 0 & -1 & 0 \\ 0 & 0 & 0 \end{matrix}\right), \qquad
    \bs{\Phi}_2 
            := \frac{1}{\sqrt{6}} \, \left(\begin{matrix} 1 & 0 & 0 \\ 0 & 1 & 0 \\ 0 & 0 & -2 \end{matrix}\right), \qquad
    \bs{\Phi}_3 
            := \frac{1}{\sqrt{2}} \, \left(\begin{matrix} 0 & 1 & 0 \\ 1 & 0 & 0 \\ 0 & 0 & 0 \end{matrix}\right), \\
    \bs{\Phi}_4 
            &:= \frac{1}{\sqrt{2}} \, \left(\begin{matrix} 0 & 0 & 1 \\ 0 & 0 & 0 \\ 1 & 0 & 0 \end{matrix}\right), \qquad 
    \bs{\Phi}_5 
            := \frac{1}{\sqrt{2}} \, \left(\begin{matrix} 0 & 0 & 0 \\ 0 & 0 & 1 \\ 0 & 1 & 0 \end{matrix}\right).
\end{align*}

Furthermore, let $\,\vartheta_1,\ldots,\vartheta_M : \Omega\longrightarrow\RR\,$ be such that $\, \mv{B}_{V} := \big\lbrace \mv{e}_k \, \vartheta_i \; ; \; (k,i)\in\ul{d}\times\ul{M} \big\rbrace$ forms a basis of $V_{hp}$. Then, by representing $(\mv{v}_{hp},\bs{q}_{hp},\bs{\mu}_{hp})\in\VV_{hp}\times\Lambda_{hp}$ as 
\begin{align*}
	\mv{v}_{hp} &= \sum_{i\in\ul{M}} \sum_{k\in\ul{d}} v_{d(i-1)+k} \, \mv{e}_k \, \vartheta_i, \qquad
	\bs{q}_{hp} = \sum_{i\in\ul{N}} \sum_{k\in\ul{L}} q_{L(i-1)+k} \, \bs{\Phi}_k \, \phi_i, \\
	\bs{\mu}_{hp} &= \sum_{i\in\ul{N}} \sum_{k\in\ul{L}} \mu_{L(i-1)+k} \, \bs{\Phi}_k \, \varphi_i
\end{align*}
any $(\mv{v}_{hp},\bs{q}_{hp},\bs{\mu}_{hp})\in\VV_{hp}\times\Lambda_{hp}$ is completely determined by the coefficient vectors 
%
\begin{align*}
	\mv{a} := (v_1,\ldots,v_{dM})^{\top}\in\RR^{dM}, \qquad
	\mv{b} := (q_1,\ldots,q_{LN})^{\top}, \,
	\mv{c} := (\mu_1,\ldots,\mu_{LN})^{\top}\in\RR^{LN}.
\end{align*}
The coefficient vectors of the discrete solution of \eqref{eq:discrete_MVF} we denote by $\mv{a}^{\ast}$, $\mv{b}^{\ast}$ and $\mv{c}^{\ast}$. By representing $\bs{\chi}_{i,\varrho}(\bs{q}_i,\bs{\mu}_i)$ for all $i\in\ul{N}$ as $\, \bs{\chi}_{i,\varrho}(\bs{q}_i,\bs{\mu}_i) = \sum_{k\in\ul{L}} \chi_{i,k} \, \bs{\Phi}_k\,$ it is pointed out in \cite{ref:Bammer2026newtonSolver} that the coefficients $\chi_{i,1},\ldots,\chi_{i,L}$ result in 
\begin{align*}
	\chi_{i,l} = \max\big\lbrace \sigma_i, \abs{\mv{c}_i+\varrho \, \mv{b}_i} \big\rbrace \, \mu_{L(i-1)+l} - \sigma_i \, \big( \mu_{L(i-1)+l} + \varrho \, q_{L(i-1)+l} \big)
\end{align*}
with $\mv{b}_i,\mv{c}_i\in\RR^L$ given by $\mv{b}_i := (q_{L(i-1)+1},\ldots,q_{Li})^{\top}$ and $\mv{c}_i := (\mu_{L(i-1)+1},\ldots,\mu_{Li})^{\top}$ for $i\in\ul{N}$. Next, we define 
\begin{align*}
	\mv{S}_{i,\varrho} : \RR^L\times \RR^L\longrightarrow\RR^L, \qquad
	\mv{S}_{i,\varrho}(\mv{b}_i,\mv{c}_i) := (\chi_{i,1},\ldots,\chi_{i,L})^{\top}
	 \qquad\forall \, i\in\ul{N}
\end{align*}
and let $\bs{A}\in\RR^{dM\times dM}$, $\bs{C},\bs{D}\in\RR^{LN\times LN}$, $\bs{B}\in\RR^{dM\times LN}$ and $\mv{l}\in\RR^{dM}$ be given by
\begin{alignat*}{2}
	A_{d(i-1)+k, \, d(j-1)+l} 
     &:= a\big( (\mv{e}_l \, \vartheta_j,\mathbf{0}) , (\mv{e}_k \, \vartheta_i,\mathbf{0}) \big)
     \qquad &&\forall \, (i,j,k,l)\in\ul{M}^2 \times\ul{d}^2, \\
	C_{L(i-1)+k, \, L(j-1)+l} 
     &:= a\big( (\mv{0},\boldsymbol{\Phi}_l \, \phi_i) , (\mv{0},\boldsymbol{\Phi}_k \, \phi_j) \big)
     \qquad &&\forall \, (i,j,k,l)\in\ul{N}^2 \times\in\ul{L}^2, \\
	D_{L(i-1)+k, \, L(j-1)+l} 
     &:= \delta_{lk} \, \delta_{ij} \, D_{i}
     \qquad &&\forall \, (i,j,k,l)\in\ul{N}^2 \times\in\ul{L}^2, \\
	B_{d(i-1)+k, \, L(j-1)+l} 
     &:= -a\big( (\mv{o},\bs{\Phi}_l \, \phi_j) , (\mv{e}_k \, \vartheta_i, \mathbf{0}) \big)
     \qquad &&\forall \, (i,j,k,l)\in\ul{N}\times\ul{M}\times\ul{d}\times\ul{L}, \\
	l_{d(i-1)+k} 
	 &:= -\ell(\mv{e}_k \, \vartheta_i)
	 \qquad &&\forall \, (i,k)\in\ul{M}\times\ul{d}.
\end{alignat*}
Note that the quadratic matrices $\bs{A}$ and $\bs{C}$ are symmetric and positive definite and $\bs{D}$ is a positive definite diagonal matrix. By defining the affine linear function
\begin{align*}
	\mv{L} : \RR^{dM}\times\RR^{LN}\times\RR^{LN}\longrightarrow\RR^K, \qquad
	\mv{L}(\mv{a},\mv{b},\mv{c}) := \left(\begin{matrix}
	 	\bs{A} & \bs{B} & \bs{0} \\
	 	\bs{B}^{\top} & \bs{C} & \bs{D}
	 \end{matrix}\right) \left(\begin{matrix}
	 	\mv{a} \\ \mv{b} \\ \mv{c}
	 \end{matrix}\right) + \left(\begin{matrix}
	 	\mv{l} \\ \mv{0}
	 \end{matrix}\right)
\end{align*}
with $K:= dM + LN$ we finally introduce the (non-smooth) function
\begin{align*}
\mv{F} : \RR^{dM}\times\RR^{LN}\times\RR^{LN}\longrightarrow\RR^{dM + 2 \, LN}, \qquad
    \mv{F}(\mv{a},\mv{b},\mv{c}) := \left(\begin{smallmatrix}
		\mv{L}(\mv{a},\mv{b},\mv{c}) \\
		\mv{S}_{1,\varrho}(\mv{b}_1,\mv{c}_1) \\
		 \vdots \\
		\mv{S}_{N,\varrho}(\mv{b}_N,\mv{c}_N)
	\end{smallmatrix}\right).
\end{align*}

Therewith, \eqref{eq:discrete_MVF} becomes equivalent to the nonlinear problem: \emph{Find $\mv{a}^{\ast}\in\RR^{dM}$ and $\mv{b}^{\ast},\mv{c}^{\ast}\in\RR^{LN}$ such that}
\begin{align}\label{eq:myNonlinearP}
	\mv{F}(\mv{a}^{\ast},\mv{b}^{\ast},\mv{c}^{\ast}) = \mv{0},
\end{align}
see \cite{ref:Bammer2022Icosahom, ref:Bammer2025apriori, ref:Bammer2026newtonSolver}. 
Since \eqref{eq:discrete_MVF} has a unique solution $(\mv{u}_{hp},\bs{p}_{hp},\bs{\lambda}_{hp})\in\VV_{hp}\times\Lambda_{hp}$ the nonlinear problem \eqref{eq:myNonlinearP} is uniquely solvable as well. By Lemma~1 in \cite{ref:Bammer2026newtonSolver}, $\mv{F}$ is Lipschitz-continuous and semismooth. To solve \eqref{eq:myNonlinearP} numerically, we introduce the semismooth Newton solver, Algorithm~\ref{eq:SSN_iterates}, where $\partial\mv{F}$ denotes the Clarke subdifferential of $\mv{F}$. This solver is analysed in \cite{ref:Bammer2026newtonSolver} and the numerical examples therein show its applicability and robustness of the number of iterations on the finite element spaces and $\varrho$, cf.~\eqref{eq:defi_chi_irho}.

%
\begin{algorithm}
\caption{Semismooth Newton Solver}\label{eq:SSN_iterates}
	\begin{algorithmic}[1]
    \Procedure{NewtonSolver} {$\mv{a}^{(0)},\mv{b}^{(0)},\mv{c}^{(0)}, \text{tol}$} \vspace{0.1cm}
    	\State $k \gets 0$ \vspace{0.1cm}
        \While{$\abs{\mv{F}\big(\mv{a}^{(k)},\mv{b}^{(k)},\mv{c}^{(k)}\big)} > \text{tol}$} \vspace{0.1cm}
        \State Choose $\bs{H}_k\in\partial\mv{F}\big( \mv{a}^{(k)}, \mv{b}^{(k)}, \mv{c}^{(k)}\big)$
        \State Solve 
        $\bs{H}_k \big( \Delta\mv{a}^{(k)}, \Delta\mv{b}^{(k)}, \Delta\mv{c}^{(k)} \big)^{\top} 
         = - \mv{F}\big( \mv{a}^{(k)}, \mv{b}^{(k)}, \mv{c}^{(k)} \big)$
        \State Set 
        $\big( \mv{a}^{(k+1)}, \mv{b}^{(k+1)}, \mv{c}^{(k+1)} \big)^{\top}
         = \big( \mv{a}^{(k)}, \mv{b}^{(k)}, \mv{c}^{(k)} \big)^{\top} + \big( \Delta\mv{a}^{(k)}, \Delta\mv{b}^{(k)}, \Delta\mv{c}^{(k)} \big)^{\top}$ \vspace{0.1cm}
        \State $k\gets k+1$ \vspace{0.1cm}
        \EndWhile \vspace{0.1cm}
        \State \textbf{return} $\big(\mv{a}^{(k)},\mv{b}^{(k)},\mv{c}^{(k)}\big)$ \vspace{0.1cm}
    \EndProcedure
	\end{algorithmic}
\end{algorithm}
\vspace{-0.25cm}

\begin{theorem}[Superlinear Convergence -- {\cite[Thm.8]{ref:Bammer2026newtonSolver}}]
	Algorithm~\ref{eq:SSN_iterates} is well-defined in a neighbourhood of the solution of $\,\mv{F}(\mv{a}^{\ast},\mv{b}^{\ast},\mv{c}^{\ast}) = \mv{0}\, $ and converges locally superlinear.
\end{theorem}

\medskip


\section{A Priori and A Posteriori Error Analysis}\label{sec:apriori_analysis}\label{sec:aprioriEst}
\smallskip

\textbf{A Priori Analysis} -- For the a priori analysis we make the additional assumption
\begin{align}\label{eq:mappingAssumpt}
	\det\nabla\mv{M}_T\in\PP_1(\wh{T})
	 \qquad\forall \, T\in\Th \, \text{ with } \, p_T\geq 2,
\end{align}
which is no restriction for $d=2$ but slightly limits the mesh elements' shape for $d=3$.

\begin{theorem}[{\cite[Thm.7]{ref:Bammer2025apriori}}]\label{prop:rep_lambdahp}
	Under the assumption \eqref{eq:mappingAssumpt}, $\Lambda_{hp}$ can be represented as 
	\begin{align*}
		\Lambda_{hp} = \Big\lbrace \bs{\mu}_{hp}\in Q_{hp} \; ; \; \vert\bs{\mu}_{hp}(\mv{M}_T(\hat{\mv{x}}_T))\vert_F \leq \sigma_y \; \forall \, k\in\ul{n_T}, \; \forall \, T\in\Th \Big\rbrace
	\end{align*}		
\end{theorem}

Note that the above representation is a discretization of \eqref{eq:lambda_strong}, for $d=3$, leading to a discrete mixed variational formulation that, in general, differs from \eqref{eq:discrete_MVF} if \eqref{eq:mappingAssumpt} does not hold true.

\begin{theorem}[{\cite[Thm.6]{ref:Bammer2025apriori}}]\label{prop:apriori_Estimate}
	Let $(\mv{u},\bs{p},\bs{\lambda})\in\VV\times\Lambda$ and $(\mv{u}_{hp},\bs{p}_{hp},\bs{\lambda}_{hp})\in\VV_{hp}\times\Lambda_{hp}$ be the unique solution of \eqref{eq:mixed_varF} and \eqref{eq:discrete_MVF}, respectively. Under assumption \eqref{eq:mappingAssumpt} there are $c_1,c_2>0$ s.t.~for all $(\mv{v}_{hp},\bs{q}_{hp},\bs{\mu}_{hp})\in\VV_{hp}\times\Lambda_{hp}$ and arbitrary $\bs{\mu}\in\Lambda$ it holds that
	\begin{align*}
		\Vert (\mv{u} - \mv{u}_{hp}, \bs{p} - \bs{p}_{hp}) \Vert^2 + \Vert \bs{\lambda} - \bs{\lambda}_{hp} \Vert_{0,\Omega}^2
		&\leq c_1 \, \Big( \Vert (\mv{u} - \mv{v}_{hp}, \bs{p} - \bs{q}_{hp})\Vert^2 + \Vert \bs{\lambda} - \bs{\mu}_{hp}\Vert_{0,\Omega}^2 \Big) \\
		& \qquad\qquad + c_2 \, \big( \bs{p}, \bs{\lambda}_{hp} - \bs{\mu} + \bs{\lambda} - \bs{\mu}_{hp} \big)_{0,\Omega}.
	\end{align*}
\end{theorem}

In the rest of the article, we let $h := \max_{T\in\Th} h_T$,  $p := \min_{T\in\Th} p_T$
and write $A\lesssim B$ if there is a constant $c>0$ \emph{independent of $h$ and $p$} such that $A \leq c \, B$. Moreover, let $\abs{\, \cdot \,}_{k,\Omega}$ denote the Sobolev seminorm on $H^k(\Omega,X)$ for $X\in\lbrace\RR^d, \RR^{d\times d}\rbrace$ and $\abs{\, \cdot \,}_{k,T}$ its local version on $T\in\Th$. The next results heavily rely on density arguments and approximation properties of the following projection/interpolation operators, see \cite{ref:Ainsworth_2000, ref:melenk_2005, ref:Sanchez_1984}.
\begin{enumerate}
	\item[$\bullet$] Let $\II_{hp} : H^1(\Omega,\RR^d) \longrightarrow V_{hp}$ be an $H^1$-projection operator which, for $s\geq 1$, satisfies
	\begin{align*}
		\Vert \mv{v} - \II_{hp}(\mv{v}) \Vert_{1,\Omega} 
		 \lesssim \frac{h^{\min\lbrace p, s-1\rbrace}}{p^{s-1}} \, |\mv{v}|_{s,\Omega}
		\qquad\forall \, \mv{v}\in H^s(\Omega,\RR^d).
	\end{align*}
	\item[$\bullet$] For $T\in\Th$ and $t\geq 0$ the standard $L^2$-projection operator $\P_{hp} : Q\longrightarrow Q_{hp}$ satisfies
	\begin{align*}
		\Vert \bs{q} - \P_{hp}(\bs{q}) \Vert_{0,T}
		 \lesssim \frac{h_T^{\min\lbrace p_T, t\rbrace}}{p_T^t} \, \abs{\bs{q}}_{t,T}
		\qquad\forall \, \bs{q}\in H^t(T,\RR^{d\times d}).
	\end{align*}
	\item[$\bullet$] Let $\Jhp : \prod_{T\in\Th} C^0(T,\RR^{d\times d})\cap Q \longrightarrow Q_{hp}$ be the nodal interpolation operator defined in \cite{ref:Bammer2025apriori}. Then, for $T\in\Th$ and $t>d/2$ it holds that
	\begin{align*}
		\Vert \bs{q} - \Jhp(\bs{q}) \Vert_{0,T}
		 \lesssim \frac{h_T^{\min\lbrace p_T, t\rbrace}}{p_T^t} \, \abs{\bs{q}}_{t,T}
		\qquad\forall \, \bs{q}\in H^t(T,\RR^{d\times d}).
	\end{align*}
	%
\end{enumerate}

\begin{theorem}[Norm Convergence -- {\cite[Thm.8]{ref:Bammer2025apriori}}]
	Let $(\mv{u},\bs{p},\bs{\lambda})\in\VV\times\Lambda$ and $(\mv{u}_{hp},\bs{p}_{hp},\bs{\lambda}_{hp})\in\VV_{hp}\times\Lambda_{hp}$ be the solution of \eqref{eq:mixed_varF} and \eqref{eq:discrete_MVF}, respectively. Then,
	\begin{align*}
		\lim_{h/p\to 0} \Big( \Vert \mv{u} - \mv{u}_{hp} \Vert_{1,\Omega}^2 + \Vert \bs{p} - \bs{p}_{hp} \Vert_{0,\Omega}^2 + \Vert \bs{\lambda} - \bs{\lambda}_{hp} \Vert_{0,\Omega}^2 \Big) = 0.
	\end{align*}
\end{theorem}

For the derivation of convergence rates let the solution of \eqref{eq:mixed_varF} satisfy the regularity assumption
\begin{align*}
	(\mv{u},\bs{p},\bs{\lambda}) \in H^s(\Omega,\RR^d) \times H^t(\Omega,\RR^{d\times d}) \times  H^s(\Omega,\RR^{d\times d})
\end{align*}
for $s\geq 1$ and $t,l\geq 0$. Note that \eqref{eq:discrete_MVF} is conforming in the displacement and plastic strain variable but, in general, \emph{not} in the discrete Lagrange multiplier. However, if $p_T=1$ for all $T\in\Th$ the discretization is conforming in all variables so that $\Lambda_{hp}\subseteq\Lambda$ and $\P_{hp}(\Lambda)\subseteq\Lambda_{hp}$.

\begin{theorem}[Optimal Convergence -- {\cite[Thm.9]{ref:Bammer2025apriori}}]
	Let $p_T=1$ for all $T\in\Th$ and the weak solution $(\mv{u},\bs{p},\bs{\lambda})$ of \eqref{eq:mixed_varF} satisfy the above regularity assumption. Then, 
	\begin{align*}
		\Vert (\mv{u} - \mv{u}_{hp}, \bs{p} - \bs{p}_{hp}) \Vert^2 + \Vert \bs{\lambda} - \bs{\lambda}_{hp} \Vert_{0,\Omega}^2
		 \lesssim h^{2 \, \min\lbrace 1, s-1, t, l \rbrace} \, \Big( |\mv{u}|_{s,\Omega}^2 + |\bs{p}|_{t,\Omega}^2 + |\bs{\lambda}|_{l,\Omega}^2 \Big).
	\end{align*}
\end{theorem}
Since for higher-order methods the \emph{non-conformity error} $(\bs{p}, \bs{\lambda}_{hp} - \bs{\mu})_{0,\Omega}$ has to be estimated for $\bs{\mu}\in\Lambda$, cf.~Theorem~\ref{prop:apriori_Estimate}, and as $\P_{hp}(\Lambda)\nsubseteq\Lambda_{hp}$ the operator $\Jhp$ may has to be used on $T\in\Th$ with $p_T\geq 2$ one usually obtains reduced guaranteed convergence rates (which are common for higher-order mixed methods for variational inequalities).

\begin{theorem}[Convergence Rates -- {\cite[Thm.11]{ref:Bammer2025apriori}}]
	If the weak solution $(\mv{u},\bs{p},\bs{\lambda})$ of \eqref{eq:mixed_varF} satisfies the above regularity assumption for $s\geq 1$, $t,l\geq d/2$ and \eqref{eq:mappingAssumpt} holds true, then
	\begin{align*}
		\Vert (\mv{u} - \mv{u}_{hp}, \bs{p} - \bs{p}_{hp}) \Vert^2 + \Vert \bs{\lambda} - \bs{\lambda}_{hp} \Vert_{0,\Omega}^2
		 \lesssim \frac{h^{\min\lbrace p, 2s-2, t, l \rbrace}}{p^{\min\lbrace 2s-2, t, l \rbrace}}.
	\end{align*}
\end{theorem}

\medskip


\textbf{A Posteriori Error Analysis} -- Let $(\mv{u}_N,\bs{p}_N,\bs{\lambda}_N)\in\VV\times Q$ be an arbitrary triple for which we aim to derive upper and lower error bounds based on an auxiliary problem. The considerations therefore include the cases where $(\mv{u}_N,\bs{p}_N,\bs{\lambda}_N)$ represents the discrete solution of \eqref{eq:discrete_MVF} or an approximation of it (for instance, when applying an iterative solution scheme). These error bounds then allow to construct a residual-based a posteriori error estimator $\eta^2$, that is reliable and efficient.

To separate the error corresponding to the finite element approximation from the one that can be associated with the discretization of the plasticity constraints we consider the auxiliary problem: \emph{For given $\bs{\lambda}_N\in Q$, find a pair $(\mv{u}^{\ast},\bs{p}^{\ast})\in\VV$ such that}
\begin{align}\label{eq:auxRel}
	a\big( (\mv{u}^{\ast},\bs{p}^{\ast}), (\mv{v},\bs{q}) \big) = \ell(\mv{v}) - (\bs{\lambda}_N,\bs{q})_{0,\Omega} \qquad \forall \, (\mv{v},\bs{q})\in\VV.
\end{align}

Note that the unique existence of a solution $(\mv{u}^{\ast},\bs{p}^{\ast})\in\VV$ of \eqref{eq:auxRel} is guaranteed by the Lax-Milgram theorem. Moreover, subtracting \eqref{eq:weak_vareq} from \eqref{eq:auxRel} gives the relation
\begin{align*}
	a\big( (\mv{u}^{\ast} - \mv{u}, \bs{p}^{\ast} - \bs{p}), (\mv{v},\bs{q}) \big) + (\bs{\lambda}_N - \bs{\lambda}, \bs{q})_{0,\Omega} = 0 \qquad
	 \forall \, (\mv{v},\bs{q})\in\VV
\end{align*}
By defining the \emph{global plasticity error contribution}
\begin{align*}
	\ep(\bs{\mu}) := \norm{ \bs{\mu} - \bs{\lambda}_N }_{0,\Omega}^2 + \psi(\bs{p}_N) - (\bs{\mu}, \bs{p}_N)_{0,\Omega}
	 \qquad\forall \, \bs{\mu}\in\Lambda
\end{align*}
and using the above relation we arrive at:

\begin{theorem}[Upper Error Estimate -- {\cite[Thm.5]{ref:Bammer2025aposteriori}}]\label{prop:upper_errprEst}
	For all $\bs{\mu}\in\Lambda$ it holds that
	\begin{align*}
		\Vert (\mv{u} - \mv{u}_N, \bs{p} - \bs{p}_N) \Vert^2 + \Vert \bs{\lambda} - \bs{\lambda}_N \Vert_{0,\Omega}^2 
		\lesssim \Vert (\mv{u}^{\ast} - \mv{u}_N, \bs{p}^{\ast} - \bs{p}_N) \Vert^2 + \ep(\bs{\mu}).
	\end{align*}
\end{theorem}

To obtain a lower bound we let
\begin{align*}
	\wh{\bs{\mu}} := \bs{\lambda}_N + \frac{1}{2} \, \bs{p}_N \in Q.
\end{align*}
Then, its $L^2$-projection onto $\Lambda$ is given by 
\begin{align*}
	\bs{\mu}^{\ast} := \min \big\lbrace 1, \, \sigma_y \abs{ \wh{\bs{\mu}} }_F^{-1} \big\rbrace \, \wh{\bs{\mu}},
\end{align*}
see \cite{ref:Bammer2025aposteriori}.

\begin{lem}[Minimizer of $\ep$ -- {\cite[Lem.6, Lem.16]{ref:Bammer2025aposteriori}}]
    The element $\bs{\mu}^{\ast}\in\Lambda$ uniquely minimizes $\ep(\cdot)$ over $\Lambda$. Moreover, the following estimate holds true
		\begin{align}\label{eq:estimate_eplast}
			\ep(\bs{\mu}^{\ast})
			 \lesssim \Vert \bs{\lambda} - \bs{\lambda}_N \Vert_{0,\Omega}^2 + \Big( \Vert\sigma_y\Vert_{0,\Omega} + \Vert\bs{\lambda}\Vert_{0,\Omega} \Big) \, \Vert \bs{p} - \bs{p}_N \Vert_{0,\Omega}.
		\end{align}
\end{lem}

Note that \eqref{eq:estimate_eplast} contains the linear term $\Vert \bs{p} - \bs{p}_N\Vert_{0,\Omega}$ which also appears in the next Theorem (typical for a posteriori error estimates in the context of elastoplasticity).

\begin{theorem}[Lower Error Estimate -- {\cite[Lem.7,  Thm.8]{ref:Bammer2025aposteriori}}]\label{prop:lower_errorEst}
    It holds that
    \begin{align*}
        \Vert (\mv{u}^{\ast} - \mv{u}_N, \bs{p}^{\ast} - \bs{p}_N) \Vert^2 \lesssim \Vert (\mv{u} - \mv{u}_N, \bs{p} - \bs{p}_N) \Vert^2 + \Vert \bs{\lambda} - \bs{\lambda}_N \Vert_{0,\Omega}^2.
    \end{align*}
    Moreover, for the minimizer $\bs{\mu}^{\ast}\in\Lambda$ of $\ep(\cdot)$ we find that
		\begin{align*}
			\Vert (\mv{u}^{\ast} - \mv{u}_N, \bs{p}^{\ast} - \bs{p}_N) \Vert^2 &+ \ep(\bs{\mu}^{\ast})
			 \lesssim \Vert (\mv{u} - \mv{u}_N, \bs{p} - \bs{p}_N) \Vert^2 + \Vert \bs{\lambda} - \bs{\lambda}_N \Vert_{0,\Omega}^2 + \Vert \bs{p} - \bs{p}_N \Vert_{0,\Omega}.
		\end{align*}
\end{theorem}

Under additional assumptions on $\bs{p}_N$ and $\bs{\lambda}_N$ (particularly including the lowest-order case) the linear term $\Vert \bs{p} - \bs{p}_N \Vert_{0,\Omega}$ and $\ep$ \emph{do not} appear in the error estimates of Theorem~\ref{prop:upper_errprEst} and \ref{prop:lower_errorEst}, see \cite[Cor.~9]{ref:Bammer2025aposteriori}. The last two theorems also allow us to construct a residual-based a posteriori error estimator $\eta^2$, that is reliable and efficient. We refer to \cite{ref:Bammer2025aposteriori} for a detailed construction and the proofs of reliability and efficiency.


\newpage


%
%

\end{document}